\documentclass[a4paper,11pt]{article}
\usepackage[utf8]{inputenc}
\usepackage{amsmath, amssymb, amsthm}
\usepackage{mathrsfs}
\usepackage{dsfont}
\usepackage{hyperref}
\usepackage{graphicx}
\usepackage{geometry}
\usepackage{enumitem} 
\numberwithin{equation}{section} 

\title{Uniqueness and Longtime Behavior of the Completely Positively Correlated Symbiotic Branching Model}
\author{Eran Avneri, Leonid Mytnik}
\date{}
\newtheorem{theorem}{Theorem}[section]
\newtheorem{lemma}[theorem]{Lemma}
\newtheorem{definition}[theorem]{Definition}
\newtheorem{convention}[theorem]{Convention}
\newtheorem{remark}[theorem]{Remark}
\newtheorem{corollary}[theorem]{Corollary}
\newtheorem{proposition}[theorem]{Proposition}

\begin{document}

\maketitle
\begin{abstract}
The symbiotic branching model in $\mathbb{R}$ describes the behavior of two branching populations migrating in space $\mathbb{R}$ in terms of a corresponding system of stochastic partial differential equations. The system is parametrized with a correlation parameter $\rho$, which takes values in $[-1,1]$ and governs the correlation between the branching mechanisms of the two populations. While existence and uniqueness for this system were established for $\rho \in [-1,1)$, weak uniqueness for the completely positively correlated case of $\rho = 1$ has been an open problem. In this paper, we resolve this problem, establishing weak uniqueness for the corresponding system of stochastic partial differential equations. The proof uses a new duality between the symbiotic branching model and the well-known parabolic Anderson model. Furthermore, we use this duality to investigate the long-term behavior of the completely positively correlated symbiotic branching model. We show that, under suitable initial conditions, after a long time, one of the populations dies out. We treat the case of integrable initial conditions and the case of bounded non-integrable initial conditions with well-defined mean. 

\end{abstract}

\section{Introduction}
\subsection{Symbiotic Branching Model}
Consider the following system of stochastic partial differential equations:

\begin{equation}
\label{eq: general}
\begin{aligned}
\begin{cases}
    \frac{\partial U^1_t(x)}{\partial t} &= \frac{1}{2}\Delta U^1_t(x) + \sqrt{U^1_t(x)U_t^2(x)} \dot{W_1}(t,x), \quad x \in \mathbb{R}, t \geq 0,  \\
    \frac{\partial U^2_t(x)}{\partial t} &= \frac{1}{2}\Delta U^2_t(x) + \sqrt{U^1_t(x)U_t^2(x)} \dot{W_2}(t,x), \quad x \in \mathbb{R}, t \geq 0, \\
\end{cases}
\end{aligned}
\end{equation}
where $\dot{W_1}$ and $\dot{W_2}$ are $\rho$-correlated Gaussian white noises on $\mathbb{R}_+ \times \mathbb{R}$ and $\Delta$ is the Laplacian operator.
This system was introduced  by \cite{MR2094150} and is called the \textit{symbiotic branching model} (SBM). Let us mention that this model was studied extensively in the literature, see \cite{MR4331866}, \cite{MR2778802}, \cite{MR3846839}, \cite{MR3474460}, \cite{MR4801607}, \cite{hammer2015infinite}, \cite{MR2814415}, \cite{MR4076768}. It is worth noticing that the SBM is an extension of the so-called mutually catalytic branching model, that also received a lot of attention, see, \cite{MR1653845}, \cite{MR3111225}, \cite{MR1944004}, \cite{MR1921744}, \cite{MR1959845}, \cite{MR1634416}, \cite{MR2663642}, \cite{MR2642883}. It is only natural to ask whether uniqueness in law holds for $(\ref{eq: general})$. It was shown in \cite{MR1653845} that for $\rho =0$ uniqueness in law holds for $(\ref{eq: general})$, for a large family of initial conditions. The proof used a self-duality technique. The self-duality method was later generalized in \cite{MR2094150}, proving uniqueness in law for $\rho \in (-1,1)$. However, this self-duality becomes trivial for $|\rho| = 1$ and can no longer be used to show uniqueness in these cases. Uniqueness in law for the case $\rho = -1$ follows from the particle system moment duality (see Proposition 12 in \cite{MR2094150} for details). However, the moment duality does not give uniqueness in the case of $\rho = 1$. In the case of $\rho = 1$ and for initial conditions $U_0^1=U_0^2$, the SBM coincides with the parabolic Anderson model, for which uniqueness is well-known. For general initial conditions, uniqueness in law for $\rho = 1$ has been a long-standing open problem. For this case, we get the following system of stochastic partial differential equations:


\begin{equation}
\label{eq: formal}
\begin{aligned}
\begin{cases}
    \frac{\partial U^1_t(x)}{\partial t} &= \frac{1}{2}\Delta U^1_t(x) + \sqrt{U^1_t(x)U_t^2(x)} \dot{W}(t,x),  \quad x \in \mathbb{R}, t \geq 0,  \\
    \frac{\partial U^2_t(x)}{\partial t} &= \frac{1}{2}\Delta U^2_t(x) + \sqrt{U^1_t(x)U_t^2(x)} \dot{W}(t,x),  \quad x \in \mathbb{R}, t \geq 0, \\
\end{cases}
\end{aligned}
\end{equation}
where $\dot{W}$ is a white noise on $\mathbb{R}_+ \times \mathbb{R}$. In this paper, we are interested in establishing weak uniqueness for $(\ref{eq: formal})$. To be more precise, we will show uniqueness to the martingale problem corresponding to  $(\ref{eq: formal})$. We will use the duality technique to prove uniqueness. As we have mentioned, the self duality introduced in \cite{MR2094150} is not useful in the case of $\rho = 1$, so we introduce a new duality function, with respect to the parabolic Anderson model, which has also been the subject of extensive research, see \cite{MR4700256}, \cite{MR3729613}, \cite{MR3414455},\cite{MR4334682}. This new duality relation, allows us not only to establish uniqueness in law for a wide range of initial conditions, but also enables us to investigate the long-term behavior of the unique solution to $(\ref{eq: formal})$. \\
Note that the SBM arises as the scaling limit of a particle system, see Remark 1 in \cite{MR2094150} for details.  Thus, if $(U^1,U^2)$ is a solution to $(\ref{eq: formal})$, then we may think of $U^i_t(x)$ as the density of population $i$ at time $t$ at the point $x$. Our key question in the long-term analysis is whether both populations will survive forever or at least one population will die out. If there is a positive probability that both populations will survive, we say that coexistence is possible; otherwise, we say coexistence is impossible. The question of the long-term behavior of SBM has been studied in the literature. It was proved in \cite{MR1634416} that for SBM with $\rho = 0$, coexistence is impossible, for rapidly decreasing and flat initial conditions. This non-coexistence result was later proven for a more general class of initial conditions, which are not necessarily flat or integrable, see Theorem 4.2 in \cite{MR1765002}. The non-coexistence result was also extended to all $\rho \in (-1,1)$ (see Proposition 2.1 in \cite{MR2778802}).

In addition, there are known results regarding coexistence for the discrete-space version of SBM, where the state space $\mathbb{R}$ is replaced with $\mathbb{Z}^d$. For $\rho \in (-1,0)$, it has been shown that coexistence is possible if and only if $d \geq 3$ (see Theorem 2.1 in \cite{MR3111225} or Theorem 1.2 in \cite{MR1634416}). For $\rho \in (0,1)$, it has been proven that if $d \leq 2$, then coexistence is impossible (see Proposition 2.1 in \cite{MR2778802}). Furthermore, it has been recently proved that for $\rho = 1$, if $d \leq 2$ then coexistence is impossible, and if $d \geq 3$, coexistence is possible if and only if the branching rate is small enough (for details, see Theorem 1.6 and Theorem 1.7 in \cite{MR4801607}). The proof of the long-term coexistence (or non-coexistence) for $\rho \in (-1,1)$ usually goes through the following procedure. First, coexistence (or non-coexistence) is established for finite mass initial conditions. Then, the \textit{self duality} is used to study the long-term behavior of SBM with \textit{flat} initial conditions. In the absence of useful self-duality, this procedure does not work for obtaining the results for flat initial conditions. For example, the proof of longtime coexistence of SBM in the recurrent regime for $\rho \in (-1,1)$ crucially uses duality to get the result for flat initial conditions. In the case of $\rho = 1$, the proof in \cite{MR4801607} uses a non-trivial stochastic argument to show non-coexistence for finite mass initial conditions, which we do not see how can be adapted to the continuous state space $\mathbb{R}$. Our new duality formula allows us to show non-coexistence for both integrable and some non-integrable non-flat initial conditions. We consider it an important and new tool and believe that it can bring new ideas on how to deal with other important cases, for example, studying coexistence and non-coexistence properties of SBM in the discrete state space for $\rho \in (0,1)$ in dimensions $d \geq 3$.

\subsection{Preliminaries: notation and spaces}
Denote by $\mathscr{B}$ the set of Borel-measurable functions on $\mathbb{R}$. Let $\mathscr{B}_b$ (respectively, $C$, $C_b$) be the set of bounded (respectively, continuous, bounded continuous) Borel measurable functions on $\mathbb{R}$. In general, if $F$ is a set of functions on $\mathbb{R}$, we write $F^+$ or $F_+$ for non-negative functions in $F$. We denote by $\bold{1}$ the constant function on the real line that takes value $1$ for all $x \in \mathbb{R}$. For $f,g\in\mathscr{B}$, let $$\langle f, g \rangle = \int_{\mathbb{R}}f(x)g(x)\mathrm{d}x,$$
whenever $fg$ is integrable or non-negative.  If $\lambda \in \mathbb{R}$ and $f \in C$, we define $\left| f \right|_{\lambda} = \sup_{x \in \mathbb{R}}e^{\lambda |x|}|f(x)|$. If $\lambda = 0$, we use the more common notation $\lVert f \rVert_{\infty} = |f|_0.$  We will also use the following spaces of functions:
$$C_{\lambda} \equiv \left \{f \in C : |f|_{\lambda}  < \infty \right \},$$
$$C_{\text{exp}} \equiv \bigcup_{\lambda > 0} C_{\lambda},$$
$$C_{\text{tem}} \equiv \bigcap_{\lambda < 0} C_{\lambda},$$
$$C_{\text{rap}} \equiv \bigcap_{\lambda > 0} C_{\lambda}.$$
The topology on $C_{\text{tem}}$ (respectively, $C_{\text{rap}}$) is the locally convex topology induced by the norms $\{|f|_{\lambda} : \lambda < 0\}$  (respectively, by the norms $\{|f|_{\lambda} : \lambda > 0\}$). This topology is metrizable, and both spaces are Polish spaces with this topology.  In broad terms, the functions in $C_{\text{rap}}$ exhibit a decay rate faster than any exponential, while those in $C_{\text{tem}}$ are permitted to have at most subexponential growth. We introduce the following space of measures on $\mathbb{R}$:
$$M_{\text{tem}} = \left \{ \mu : \mu \text{ is a measure on } (\mathbb{R}, \mathscr{B}) \text{ such that} \int_{\mathbb{R}} f \mathrm{d}\mu < \infty \quad \forall f \in C_{\text{exp}}^+ \right \}.$$
We can equip $M_{\text{tem}}$ with a metric that makes it a Polish space, such that $\lim_{n \to \infty} \mu_n = \mu$ in $M_{\text{tem}}$ if and only if $\lim_{n \to \infty}\int_{\mathbb{R}} f \mathrm{d}\mu_n =\int_{\mathbb{R}} f \mathrm{d}\mu $ for all $f \in C_{\text{exp}}$. Note that $C_{\text{tem}}^+$ can be embedded in $M_{\text{tem}}$ by viewing $g \in C_{\text{tem}}^+$ as the measure $A \to \int_{A} g(x) \mathrm{d}x$.  For more details, see \cite{MR1634416}.
\\ \\
We will also use the following function space:
$$\mathcal{M} \equiv \left \{f \in C_b : \lim_{L \to \infty} \frac{1}{2L}\int_{-L}^Lf(x)\mathrm{d}x \text{ exists } \right \}.$$
Intuitively, this is the set of all bounded continuous functions with a well-defined mean on $\mathbb{R}$. For $f \in \mathcal{M}$, we denote by $\overline{f}$ its mean, that is,
\begin{equation}
\label{eq:barf}
\overline{f} = \lim_{L \to \infty}\frac{1}{2L}\int_{-L}^L f(x) \mathrm{d}x.
\end{equation}
For a metric space $E$, let $\mathscr{P}(E)$ be the set of Borel probability measures on $E$ and $C_E[0, \infty]$ be the space of continuous $E$-valued functions on $[0, \infty)$ with compact-open topology. \\ For a measure space $(X, \Sigma, \mu)$, denote by $H$ the set of all Borel measurable functions $f: X \to \mathbb{R}$, and define for $p \in [1, \infty)$:
$$L^p(X, \Sigma, \mu) \equiv \left \{f \in H : \int_{X} |f|^p \mathrm{d}\mu < \infty  \right\}.$$
Whenever the measure $\mu$ and the sigma-field $\Sigma$ are clear from context, we simply write $L^p(X)$ for $L^p(X, \Sigma, \mu)$.\\ \\
We also use the following notation regarding the heat equation. Denote the fundamental solution to the heat equation on $\mathbb{R}$ by: $$p_t(x) = \frac{1}{\sqrt{2\pi t}}e^\frac{-x^2}{2t}, \quad x \in \mathbb{R}, t > 0.$$ Let $(S_t)_{t \geq 0}$ be the corresponding heat semigroup, which acts on $f: \mathbb{R} \to \mathbb{R}$ as:

$$S_t(f)(x) = S_tf(x) = \int_{-\infty}^{\infty}p_t(x-y)f(y)\mathrm{d}y.$$ Recall that $S_t(f)$ solves $u_{t} = \frac{1}{2}u_{xx}$ with initial condition $u(0,\cdot) = f(\cdot)$. Among other things, the heat semigroup is helpful for defining a proper notion of a solution to stochastic partial differential equations, such as $(\ref{eq: formal})$. In order to define solutions, we follow the theory of Walsh (see \cite{MR876085}).  
Let $\sigma: \mathbb{R} \to \mathbb{R}$. We say that a stochastic process $(u_t(x))_{t \geq 0, x\in \mathbb{R}}$ is a solution to the stochastic partial differential equation:

\begin{equation}
\notag
    \frac{\partial u_t(x)}{\partial t} = \frac{1}{2}\Delta u_t(x) + \sigma(u_t(x)) \dot{W}(t,x),  \quad x \in \mathbb{R}, t \geq 0,
\end{equation}
if $(u_t(x))_{t \geq 0, x\in \mathbb{R}}$ satisfies the following integral equation:

\begin{equation}
\label{meaning_of_spde}
    u_t(x) = S_t(u_0)(x) + \int_0^t \int_{\mathbb{R}} p_{t-s}(x-y)\sigma(u_s(y)) \dot{W}(\mathrm{d}s, \mathrm{d}y), \quad x \in \mathbb{R}, t \geq 0.
\end{equation}
Here $\dot{W}$ is a white noise on $\mathbb{R}_+ \times \mathbb{R}$ and the double integral in $(\ref{meaning_of_spde})$ is a stochastic integral in the sense of Walsh (see \cite{MR876085}). Note that the above shorthand writing for integral equations as differential ones is extended in the obvious way to systems of equations. 

\section{Main Results}
Let $(\Omega, \mathcal{F}, \left( \mathcal{F}_t \right)_{t \geq 0}, \mathbb{P})$ be a filtered probability space, with $\left( \mathcal{F}_t \right)_{t \geq 0}$ satisfying the usual assumptions. Let $(U^1,U^2)$ be an $\mathcal{F}_t$-adapted, continuous $ C_{\text{tem}}^+\times C_{\text{tem}}^+$-valued solution to~(\ref{eq: formal}). To introduce our uniqueness result, let us note that it is a standard technique to reformulate stochastic partial differential equations in terms of corresponding martingale problems (see, for example, Section 2.1 in \cite{MR1296425}, and Section 2.2 in \cite{MR2094150}). Follow similar steps, we get that  $(U^1,U^2)$ 
 satisfies the following martingale problem:
\begin{equation}
\notag
\label{eq:Mart_prob_U}
\begin{aligned} (MP)_U
\begin{cases}
    \text{For each $T>0, \phi\in C_{\text{rap}}$}, i=1,2, \\
    M_t^{\phi, T} = \langle S_{T-t}U^i_t, \phi \rangle - \langle S_TU_0^i, \phi\rangle  \\
    \text{is a continuous square integrable $\mathcal{F}_t$-martingale on $[0,T]$ such that:  $ M_0^{\phi, T}=0$,} \\
    \langle M^{\phi, T} \rangle_{t} = \int_0^t \langle U^1_sU^2_s, \left(S_{T-s}\phi\right)^2 \rangle \mathrm{d}s.\\
\end{cases}
\end{aligned}
\end{equation}
Note that \( M^{\phi,T} \) is independent of \( i \) since both martingales have the same quadratic variation and initial condition. Thus, to show uniqueness in law for~(\ref{eq: formal}), it is enough to show uniqueness for the martingale problem $(MP)_U$.

To prove the uniqueness result, we use a one-to-one linear transformation of $U$. Define
$$X_t = U^1_t + U^2_t; \quad Y_t = U^1_t - U^2_t.$$
It follows from the definition of $(X,Y)$ and from $(\ref{eq: formal})$ that the pair $(X,Y)$ satisfies the following system of equations:
\begin{equation}
\label{eq: X and Y}
\begin{aligned}
\begin{cases}
    \frac{\partial X_t(x)}{\partial t} &= \frac{1}{2}\Delta X_t(x) + \sqrt{X^2_t(x) - Y^2_t(x)} \dot{W}(t,x),  \quad x \in \mathbb{R}, t \geq 0,  \\
    \frac{\partial Y_t(x)}{\partial t} &= \frac{1}{2}\Delta Y_t(x),  \quad x \in \mathbb{R}, t \geq 0. \\
\end{cases}
\end{aligned}
\end{equation} 
Note that $Y_t$ is deterministic given the deterministic initial conditions, hence it is sufficient to show uniqueness in law for $X$. Since $U$ satisfies $(MP)_U$, it is easy to see that $X$ with $X_0 
\in C_{
\text{tem}}^+$, satisfies the following martingale problem:
\begin{equation}
\notag
\label{eq: basic_Mart_prob}
\begin{aligned} (MP)_X
\begin{cases}
    \text{Let $Y_0 \in C_{\text{tem}}$, $\lvert Y_0 (\cdot)\rvert \leq X_0(\cdot)$. For each $T>0, \phi \in C_{\text{rap}}$} \\
    M_t^{\phi, T} = \langle S_{T-t}X_t, \phi \rangle - \langle S_TX_0, \phi \rangle  \\
    \text{is a continuous square integrable $\mathcal{F}_t$-martingale on $[0,T]$ such that:} M_{0}^{\phi, T} = 0, \\ \langle M^{\phi, T}  \rangle_{t} = \int_0^t \langle X_s^2 - Y_s^2, \left(S_{T-s}\phi\right)^2 \rangle \mathrm{d}s, \\
    \text{where } Y_s = S_sY_0, \forall s \geq 0.
\end{cases}
\end{aligned}
\end{equation}
\begin{remark}
    If $Y_0$ is random, $Y_t \in \mathcal{F}_0$ since $Y_t = S_t(Y_0)$ and $Y_0 \in \mathcal{F}_0$.
\end{remark}
\noindent
Our first main result is the duality formula, which establishes a new, previously unknown connection between a solution of the martingale problem $(MP)_X$ and the dual process, the parabolic Anderson model, which is the solution to the following equation:

\begin{equation}
\label{eq: dual}
    \frac{\partial V_t(x)}{\partial t} = \frac{1}{2}\Delta V_t(x) + V_t(x) \dot{W}(t,x),  \quad x \in \mathbb{R}, t \geq 0.
\end{equation}
\noindent
\begin{convention}
    We always choose a solution of $(\ref{eq: dual})$ to be independent of a solution to $(\ref{eq: formal})$.
\end{convention}
\noindent
Existence and pathwise uniqueness of $(\ref{eq: dual})$, with $V_0 = \phi \in C_{\text{tem}}^+$ and with sample paths in $C_{C_{\text{tem}}^+}[0, \infty)$, follows by an application of Theorem 2.2 in  \cite{MR1271224}. Whenever $V_0 = \psi \in C_{\text{rap}}^+$, existence and pathwise uniqueness of a solution with sample paths in $C_{C_\text{rap}^+}[0, \infty)$, follows by an application of Theorem 2.5 in  \cite{MR1271224}. We are now ready to state our duality formula.

\begin{theorem}
\label{prop: duality}Let $\mu \in \mathscr{P}(C_{\text{tem}}^+ \times C_{\text{tem}}^+)$ with compact support. Let $(U^1_t, U^2_t)_{t \geq 0}$ be a solution of the martingale problem $(MP)_U$ with initial distribution $\mu$ and sample paths in $C_{C_{\text{tem}}^+ \times C_{\text{tem}}^+}[0,\infty)$. Let $\phi \in C_{\text{rap}}^+$ and let $(V_t)_{t \geq 0}$ be the unique $C_{\text{rap}}^+$-valued solution of $(\ref{eq: dual})$ with $V_0 = \phi$. Then, for each $T > 0$, the following duality formula holds:
\begin{equation}
\label{duality_formula}
    \mathbb{E}\left[e^{-\langle X_T, \phi \rangle}\right] = \mathbb{E}\left[e^{-\langle X_0, V_T \rangle -\frac{1}{2}\int_0^T \langle V_r^2, \left(Y_{T-r} \right)^2 \rangle \mathrm{d}r}\right],
\end{equation}
where $X_t = U^1_t + U^2_t, Y_t = U^1_t - U^2_t = S_t(U^1_0 - U^2_0).$
\end{theorem}
\noindent
With Theorem 2.3 at hand, we immediately have the uniqueness result for $(MP)_U$.


\begin{theorem}
\label{main_res: uniqueness}
    Let $\nu \in \mathscr{P}(C_{\text{tem}}^+ \times C_{\text{tem}}^+)$. Then any two solutions of the martingale problem $(MP)_U$, with initial distribution $\nu$ and sample paths in $C_{C_{\text{tem}}^+ \times C_{\text{tem}}^+}[0,\infty)$, have the same finite dimensional distributions, that is, uniqueness in law holds for the martingale problem $(MP)_U$.
\end{theorem}
{\it Proof of Theorem~\ref{main_res: uniqueness}}:
By Theorem $\ref{prop: duality}$ and one-to-one correspondence between $(X,Y)$ and $(U^1,U^2)$,  uniqueness in law for the martingale problem  $(MP)_U$ follows easily by Proposition 4.4.7 of \cite{MR838085}. Although the boundedness condition required in Proposition 4.4.7 of \cite{MR838085} is not met here, it is not necessary and can be lifted without further adjustments to the proof. 
\qed
\begin{remark}
It follows from Theorem $\ref{main_res: uniqueness}$ and the existence result in Theorem 4 in $\cite{MR2094150}$ that
for any $(\phi, \psi) \in C_{tem}^+ \times C_{tem}^+$, 
the system $(\ref{eq: formal})$ with $(U^1_0, U^2_0) = (\phi, \psi)$ is well-posed.
\end{remark}
\noindent
Theorem $\ref{main_res: uniqueness}$ states that solutions are unique in law, so we may investigate the longtime behavior of the unique solution. The main question in our longtime behavior analysis, is whether both populations can survive forever or whether at least one population will die out. In other words, we want to check if coexistence of the populations is possible. To formally define coexistence in the case of integrable initial conditions, we need the following lemma.

\begin{lemma}
\label{sol_is_mart}
Let $\phi, \psi \in L^1(\mathbb{R}) \cap C_{\text{tem}}^+$. Let $(U^1,U^2)$ be the solution of $(\ref{eq: formal})$ with initial conditions $(U^1_0, U^2_0) = (\phi, \psi)$. Then, for $i \in \{1,2 \}$, the process $\left( \left \langle U_t^i, \bold{1} \right \rangle \right)_{t \geq 0}$ is a non-negative square integrable martingale.
\end{lemma}
\begin{proof}
    Applying the stochastic Fubini theorem, along with some simple moment estimations, we get the desired result, by Theorem 2.5 in \cite{MR876085}. We leave the details to the reader.
\end{proof}
\noindent
By the above lemma, and the martingale convergence theorem, there exist almost sure limits:
$$\lim_{t \to \infty} \left \langle U_t^1, \bold{1} \right \rangle, \quad \lim_{t \to \infty}\left \langle U_t^2, \bold{1} \right \rangle.$$
\noindent
Thus, the following definition is in place.

\begin{definition}
\label{def: integrable_coexist}
Let $\phi, \psi \in L^1(\mathbb{R}) \cap C_{\text{tem}}^+$. Let $(U^1,U^2)$ be the solution of $(\ref{eq: formal})$ with initial conditions $(U^1_0, U^2_0) = (\phi, \psi)$. We say that global coexistence is possible if $\mathbb{P}\left(\lim_{t \to \infty} \langle U^1_t, \bold{1} \rangle \langle U^2_t, \bold{1} \rangle > 0 \right) > 0$. Otherwise, we say that global coexistence is impossible.
\end{definition} 
\noindent
We are interested to define coexistence in the case of not necessarily integrable initial conditions. To characterize longtime behavior, the previous martingale convergence argument fails, hence we use limit points in $M_{\text{tem}}$. To show existence of limit points, we need to show the relevant set of measures is tight. This is done in the next lemma.

\begin{lemma}
\label{local_coexistence_is_ok}
Let $\phi, \psi \in \mathcal{M} \cap C_{\text{tem}}^+$. Let $(U^1,U^2)$ be the solution of $(\ref{eq: formal})$ with initial conditions $(U^1_0, U^2_0) = (\phi, \psi)$. Then, for $i \in \{1,2\}$ the collection $\{U^i_t\}_{t \geq 0}$, is tight in $M_{\text{tem}}$. 
\end{lemma} 
\noindent
We postpone the proof of Lemma $\ref{local_coexistence_is_ok}$ to Section 3.2. We are now ready to introduce the concept of local coexistence:
\begin{definition}
\label{def: mean_coexist}
Let $\phi,\psi \in C_{\text{tem}}^+ \cap \mathcal{M}$ and let $(U^1,U^2)$ be the solution of $(\ref{eq: formal})$ with initial conditions $(U^1_0, U^2_0) = (\phi, \psi)$. We say that local coexistence is possible if for any limit point $(U^1_{\infty}, U^2_{\infty}) \in M_{\text{tem}}^2$, and for all $g_1, g_2 \in C_{\text{exp}}$, $\mathbb{P}\left(\langle U^1_\infty, g_1 \rangle \cdot \langle U^2_\infty, g_2 \rangle > 0 \right) > 0$. Otherwise, we say that local coexistence is impossible.
\end{definition}
\noindent
With these definitions in hand, we are ready to present our main results regrading the longtime behavior of the completely positively correlated SBM.
\begin{theorem}
\label{main_res: global_longtime}
    Let $(\phi, \psi) \in C_{\text{tem}}^+ \cap L^1$. Assume that $(U^1,U^2)$ is the solution of $(\ref{eq: formal})$ with initial conditions $(U^1_0, U^2_0) = (\phi, \psi)$. Assume without loss of generality that $\langle \phi, \bold{1} \rangle \geq \langle \psi, \bold{1} \rangle$. Then, $\langle U_t^2, \bold{1} \rangle \xrightarrow[t \to \infty]{\text{a.s.}} 0 $ and $\langle U_t^1, \bold{1} \rangle \xrightarrow[t \to \infty]{\text{a.s.}} \langle \phi, \bold{1} \rangle - \langle \psi, \bold{1} \rangle$. That is, global coexistence is impossible.
\end{theorem}
\noindent
\begin{theorem}
\label{main_res: local_longtime}
Let $(\phi, \psi) \in C_{\text{tem}}^+ \cap \mathcal{M}$. Assume that $(U^1,U^2)$ is the solution of $(\ref{eq: formal})$ with initial conditions $(U^1_0, U^2_0) = (\phi, \psi)$. Assume without loss of generality that $\overline{\phi} \geq \overline{\psi}$. Then, $\langle U_t^2, g \rangle \xrightarrow[t \to \infty]{P} 0 $ and $\langle U_t^1, g \rangle \xrightarrow[t \to \infty]{P} \left( \overline{\phi} - \overline{\psi} \right) \langle g, \bold{1} \rangle$ for all $g \in C_{\text{exp}}$. In particular, $U^2_t \xrightarrow[t \to \infty]{P} 0$ in $M_{\text{tem}}$ and $U^1_t \xrightarrow[t \to \infty]{P}  \left(\overline{\phi} - \overline{\psi}\right) \bold{1}$ in $M_{\text{tem}}$. That is, local coexistence is impossible.
\end{theorem}
\noindent
Note that the above theorems not only state that coexistence is impossible, but even allow us to determine which population becomes extinct based on the initial conditions. In both cases, the population that is 'smaller' at time $t=0$ dies out, but the notion of 'smaller' differs. For integrable initial conditions, 'smaller' refers to smaller total mass, while for initial conditions with a well-defined mean, it refers to smaller average.  
\section{Proofs}
\subsection{Proof of Theorem ~\ref{prop: duality}}
\noindent
In what follows, fix arbitrary $\mu \in \mathscr{P}(C_{\text{tem}}^+ \times C_{\text{tem}}^+)$ with compact support, and let $(U^1_t, U^2_t)_{t \geq 0}$ be a solution of the martingale problem $(MP)_U$ on $(\Omega, \mathcal{F}, \left( \mathcal{F}_t \right)_{t \geq 0}, \mathbb{P})$ with initial distribution $\mu$ and sample paths in $C_{C_{\text{tem}}^+ \times C_{\text{tem}}^+}[0,\infty)$. As usual, let $X_t = U_t^1 + U_t^2, Y_t = U_t^1 - U_t^2 = S_t(U^1_0 - U^2_0)$. \\ First, we construct a martingale problem associated with the dual process $V$. Let $(\mathcal{F}_t^V)_{t \geq 0}$ be the filtration generated by $V = (V_t)_{t \geq 0}$. Similarly to the construction of the martingale problem $(MP)_U$, we observe that any $C_{\text{rap}}^+$-valued solution to $(\ref{eq: dual})$ satisfies the following martingale problem:

\begin{equation}
\notag
\label{eq: dual_Mart_prob}
\begin{aligned} (MP)_V
\begin{cases}
    \text{For each $T>0, \phi \in C_{\text{tem}}$} \\
    M_t^{\phi, T} = \langle S_{T-t}V_t, \phi \rangle - \langle S_TV_0, \phi \rangle  \\
    \text{is a continuous square integrable $\mathcal{F}_t^{V}$-martingale on $[0,T]$ such that:} \\
    \langle M^{\phi, T} \rangle_{t} = \int_0^t \langle V_s^2, \left(S_{T-s}\phi\right)^2 \rangle \mathrm{d}s.\\
\end{cases}
\end{aligned}
\end{equation}

\begin{remark}
\label{remark_test_functions}
    The martingale problem representation is helpful whenever we use It\^o's formula for an expression involving both processes $X$ and $V$, each with a different time index.
    Whenever $V_s$ is fixed, we can use it as a test function for the martingale problem $(MP)_X$. Similarly, whenever $X_t$ is fixed, we can use it as a test function for the martingale problem $(MP)_V$. This observation will be frequently used in the proof of Proposition $\ref{prop: duality}$.
\end{remark}
\begin{remark}
\label{c_rap_to_C_tem}
The martingale problem $(MP)_V$ is equivalent to $(\ref{eq: dual})$ whenever the sample paths of the solution is in $C_{\text{rap}}^+$. If the sample paths of the solution is in $C_{\text{tem}}^+$, the only difference in the corresponding martingale problem is that the test functions should be taken from $C_{\text{rap}}$. 
\end{remark}
\noindent
The following lemma will be of frequent use.
\begin{lemma}
\label{lemma: dual_total_mass_mart}
    Let $(V_t)_{t \geq 0}$ be the unique $C_{\text{tem}}^+$-valued solution to $(\ref{eq: dual})$ with $V_0 \in C_{\text{tem}}^+ \cap L^1(\mathbb{R})$. Then, for any $\varphi \in C_b^+$, and $T >0$ the process 
$\left(\left \langle S_{T-t}V_t, \varphi \right \rangle \right)_{t \in [0,T]}$ is a continuous square integrable $\mathcal{F}_t^V$-martingale. In particular, $\left (\left \langle V_t, \bold{1} \right \rangle \right)_{t \geq 0}$ is a continuous square integrable martingale.
\end{lemma}
\begin{proof}
This is easy by approximating $\varphi$ with a sequence $\varphi_n \in C_{\text{rap}}^+$ such that $\varphi_n \uparrow \varphi$ as $n \to \infty$. We leave details to the reader. For the case of $\varphi = \bold{1}$, see also the proof of Proposition 3.6 in \cite{MR4700256} or the discussion after Proposition 3.7 in \cite{MR4635721}.
\end{proof}
\noindent
We have gathered all the tools we need to prove Theorem $\ref{prop: duality}$.
\begin{proof}[Proof of Theorem 2.3]
Let $\phi \in C_{\text{rap}}^+$ and let $(V_t)_{t \geq 0}$ be the unique $C_{\text{rap}}^+$-valued solution of $(\ref{eq: dual})$ with $V_0 = \phi$. Let $T > 0$ and for $s,t \geq 0$ such that $s+t \leq T$ define \begin{equation}
\notag
    f(s,t) = \mathbb{E}\left[ e^{-\langle X_t, S_{T-t-s}V_s \rangle - \frac{1}{2}\int_0^s \langle V_r^2, Y_{T-r}^2  \rangle \mathrm{d}r} \right].
\end{equation} We show that $f(0,T) = f(T,0)$, which is $(\ref{duality_formula})$. Fix arbitrary $s,T \geq 0$. Then, for $\mathbb{P}$-a.s. $V_s$ we apply It\^o's formula to $e^{-\langle S_{T-t}X_t, V_s \rangle}$ as a function of $X$ (recall that $X$ and $V$ are independent). Thus, we get:
\begin{equation}
\label{Ito1}
    \mathrm{d}e^{-\langle S_{T-t}X_t, V_s \rangle} = \frac{1}{2}e^{-\langle S_{T-t}X_t,V_s \rangle} \langle X_t^2 - Y_t^2, (S_{T-t}V_s)^2 \rangle\mathrm{d}t + \mathrm{d}M_t^{X,V_s}, t \in [0,T].
\end{equation}
where $t \to M_t^{X,V_s}$ is an $\left(\mathcal{F}_t\right)_{t \geq 0 }$-martingale. Multiplying $(\ref{Ito1})$ by $e^{-\frac{1}{2}\int_0^s \langle V_r^2, Y^2_{T+s-r} \rangle \mathrm{d}r}$, taking expected value (recall that $Y_t = S_tY_0$ is $\mathcal{F}_0$-measurable for any $t \geq 0$) and then replacing $T$ by $T-s$, we get:
\begin{equation}
\label{abs_cont_t}
    f(s,t) = f(s,0) + \int_0^t h_1(s,r) \mathrm{d}r, \quad 0 \leq t+s \leq T,
\end{equation}
where $h_1(s,r) = \mathbb{E}\left[\frac{1}{2}e^{-\langle S_{T-s-r}X_r, V_s \rangle - \frac{1}{2}\int_0^s \langle V_u^2, Y_{T-u}^2  \rangle \mathrm{d}u}\langle (X^2_r - Y^2_r),\left(S_{T-s-r}V_s\right)^2 \rangle \right]$.
Now fix $t,T \geq 0$. Similarly, for $\mathbb{P}$-a.s. $X_t$ we apply It\^o's formula to $e^{-\langle X_t, S_{T-s}V_s \rangle - \frac{1}{2} \int_0^s \langle V_r^2,Y^2_{T+t-r} \rangle \mathrm{d}r}$ as a function of $V$. We get:
\begin{equation}
\label{Ito2}
\begin{split}
&\mathrm{d}e^{-\langle X_t, S_{T-s}V_s \rangle - \frac{1}{2} \int_0^s \langle V_r^2,Y^2_{T+t-r} \rangle \mathrm{d}r} \\ &= \frac{1}{2}e^{-\langle X_t, S_{T-s}V_s \rangle - \frac{1}{2}\int_0^s \langle V_r^2, Y_{T+t-r}^2  \rangle \mathrm{d}r}\left(\langle V_s^2,\left(S_{T-s}X_t\right)^2 - Y_{T+t - s}^2  \rangle \right) \mathrm{d}s + \mathrm{d}M_s^{V,X_t}.
\end{split}
\end{equation}
where $s \to M_s^{V,X_t}$ is an $\left(\mathcal{F}_s^V \right)_{s \geq 0 }$-martingale.
Taking expectation on $(\ref{Ito2})$ and replacing $T$ with $T-t$ we get:
\begin{equation}
\label{abs_cont_s}
    f(s, t) = f(0,t) + \int_0^s h_2(r,t) \mathrm{d}r, \quad 0 \leq t+s \leq T,
\end{equation}
where $h_2(r,t) = \mathbb{E}\left[\frac{1}{2}e^{-\langle X_t, S_{T-r-t}V_r \rangle - \frac{1}{2}\int_0^r \langle V_u^2, Y_{T-u}^2  \rangle \mathrm{d}u}\left(\langle V_r^2,\left(S_{T-r-t}X_t\right)^2  \rangle - \langle V_r^2, Y_{T - r}^2 \rangle \right) \right]$.
Note that  $(\ref{abs_cont_t})$ and $(\ref{abs_cont_s})$ shows absolute continuity in both $s$ and $t$, which allows us to apply Lemma 4.4.10 of \cite{MR838085} to get:
\begin{equation}
\label{eq: FundementalThmCalculus}
    f(t,0) - f(0,t) = \int_0^t h_1(s, t-s) - h_2(s,t-s) \mathrm{d}s,
\end{equation}
for almost every $t \in [0,T]$. Observe that $h_1(s,T-s) = h_2(s,T-s)$, so it is sufficient to show that $(\ref{eq: FundementalThmCalculus})$ is satisfied for $t=T$. This can be done through a standard approximation procedure, compare, for instance, \cite{MR1653845}.
\end{proof}
\noindent
This finishes the proof of uniqueness. Recall that a non-negative solution of $(\ref{eq: dual})$ coincides with the solution to $(\ref{eq: formal})$ with initial conditions $U_0^1 = U_0^2$. This observation yields the following immediate corollary, establishing a self-duality property for the parabolic Anderson model, which is in fact well known in the area, and will be useful for us in the longtime behavior analysis. 
\begin{corollary}
\label{cor: duality}
    Let $\dot{W}$ and $\dot{W_0}$ be two independent white noises on $\mathbb{R}_+ \times \mathbb{R}$. Let $\psi \in C_{\text{tem}}^+$ and $\phi \in C_{\text{rap}}^+$. Define $(V_t)_{t \geq 0}$ and $(\Tilde{V}_t)_{t \geq 0}$ as the unique solutions with values in $C_{\text{tem}}^+$ and $C_{\text{rap}}^+$, respectively, of equation $(\ref{eq: dual})$, driven by $\dot{W}$ and $\dot{W_0}$, respectively, with initial conditions $V_0 = \psi$  and $\Tilde{V}_0 = \phi$. Then, for each $T > 0$, the following self-duality formula holds:
\begin{equation}
\label{duality_formula_cor}
    \mathbb{E}\left[e^{-\langle \Tilde{V}_T, \psi \rangle}\right] = \mathbb{E}\left[e^{-\langle \phi, V_T \rangle }\right].
\end{equation}
\end{corollary}
\subsection{Longtime behavior. Proofs of Theorems~\ref{main_res: global_longtime}, \ref{main_res: local_longtime}.}
\label{sec:3.2}
Our analysis of the longtime behavior relies on some basic results about the longtime behavior of the solution to the heat equation, and the longtime behavior of the parabolic Anderson model, our dual process. We first prove  these results, and then Theorem $\ref{main_res: global_longtime}$ and Theorem $\ref{main_res: local_longtime}$  will follow by using our duality formula $(\ref{duality_formula})$. \\
We start with the results about the heat equation. The following proposition and the corresponding corollary state the if we start the heat equation with integrable initial condition, then the total mass of the solution is conserved over time, and as long as the initial condition has non negative total mass, then the mass of the negative part vanishes as $t \to \infty$.
\begin{proposition}
\label{heat_longtime_1}
 Let $f \in L^1(\mathbb{R})$ and let $u(t,\cdot)= S_tf(\cdot)$. Denote $M = \langle f, \bold{1} \rangle$. Then,
 \begin{enumerate}[label=(\alph*)]
     \item For all $t \geq 0$: $\langle u(t,\cdot), \bold{1} \rangle = M.$
     \item $\lim_{t \to \infty} \langle |u(t, \cdot) - Mp_t(\cdot) | , \bold{1}\rangle = 0.$
 \end{enumerate}
 
\end{proposition}

\begin{proof}
     (a) follows immediately by Fubini's Theorem. (b) is a direct consequence of (1.10) in \cite{MR1135228}, with $N=r=1$.
\end{proof}
The above proposition implies the following corollary.
\begin{corollary}
\label{cor: negative part vanishes}
    Let $u$ and $M$ as above. Assume that $M \geq 0$. Then,

    $$\lim_{t \to \infty} \langle |u(t,\cdot)|, \bold{1} \rangle = \lim_{t \to \infty} \langle u^{+}(t,\cdot), \bold{1} \rangle = M,$$
    $$\lim_{t \to \infty} \langle u^{-}(t,\cdot), \bold{1} \rangle = 0.$$
\end{corollary}

\begin{proof}
By repeating, essentially line by line, the proof of Proposition 3.6 in \cite{MR4801607}, we get the required result.
\end{proof}
\noindent
We conclude, according to Proposition $\ref{heat_longtime_1}$ and Corollary $\ref{cor: negative part vanishes}$, that whenever the initial condition for the heat equation is integrable and its total mass in non-negative, then the total mass of the solution is preserved over time, and the negative part vanishes over time. However, it is possible that the initial condition is bounded but not necessarily integrable. The following proposition gives us a uniform bound on the solution in this case.

\begin{proposition}
\label{uniform_bound_heat}
Let $f \in C_b$ and let $u(t,\cdot) = S_t(f)(\cdot)$. Then, $$\sup_{t \geq 0, x \in \mathbb{R}}\lvert u(t,x) \rvert \leq \lVert f \rVert_{\infty}.$$    
\end{proposition}
\begin{proof}
Let $M = \lVert f \rVert_{\infty}$ and note that for all $t \geq 0$ and $x \in \mathbb{R}$ we get that:
\begin{equation}
\notag
    \lvert u(t,x) \rvert = \left \lvert \int_{\mathbb{R}}p_t(x-y)f(y)\mathrm{d}y \right \rvert \leq \int_{\mathbb{R}}p_{t}(x-y)\lvert f(y) \rvert \mathrm{d}y \leq \int_{\mathbb{R}}p_{t}(x-y)M \mathrm{d}y = M.
\end{equation}
\end{proof}
\noindent
The following proposition gives us a characterization of the longtime behavior of the heat equation in the case of bounded initial datum. Recall the notation from~(\ref{eq:barf}).
\begin{proposition}
\label{convergence_of_bounded_heat}
    Let $f \in \mathcal{M}$ and let $u(t,\cdot) = S_t(f)(\cdot)$.  Then, $\lim_{t \to \infty}u(t,x) = \overline{f}$ pointwise and for every $\phi \in L^1(\mathbb{R})$, we have that $\lim_{t \to \infty} \langle u(t, \cdot), \phi \rangle = \left \langle \overline{f} \bold{1}, \phi \right \rangle.$
\end{proposition}

\begin{proof}
Define $v(t,x) = u(t, \frac{x}{\sqrt{2}}) - \overline{f}$. It is easy to see that $v$ satisfies the conditions of the theorem in \cite{MR212424}, hence we get that $\lim_{t \to \infty}v(t,x) = 0$ pointwise, which implies that $\lim_{t \to \infty}u(t,x) = \overline{f}$ pointwise, as desired. For the second part, note that by Proposition $\ref{uniform_bound_heat}$ we get the bound, $|u(t,x)\phi(x)| \leq \lVert f \rVert_{\infty} |\phi(x)| \in L^1(\mathbb{R})$, and the result follows by the Dominated Convergence Theorem. 
\end{proof}
\noindent
The above results give us what we need about the heat equation. We continue to prove results concerning the longtime behavior of the dual process.

\begin{proposition}
\label{longtime_mass_of_dual}
Let $\phi \in C_{\text{rap}}^+$ and let $(V_t)_{t \geq 0}$ be the unique $C_{\text{rap}}^+$-valued solution to $(\ref{eq: dual})$ with $V_0 = \phi$. Then, for all $f \in C_b^+$ we have $\langle V_t, f \rangle \xrightarrow[t \to \infty]{P} 0$.  
\end{proposition}
\begin{proof}
By the last display in the proof of Proposition 4.3 in \cite{MR4700256}, we have that
$$\lim_{t \to \infty} \mathbb{E}\left[ \sqrt{\langle V_t, \bold{1} \rangle} \right] = 0.$$ Denote $M = \lVert f \rVert_{\infty}$, and note that for all $\varepsilon > 0$, by the Markov inequality,
we have
\begin{equation}
\notag
    \mathbb{P}\left(|\langle V_t, f \rangle| > \varepsilon \right) \leq \mathbb{P}\left(M\langle V_t, \bold{1} \rangle > \varepsilon \right) = \mathbb{P}\left(\langle V_t, \bold{1} \rangle > \frac{\varepsilon}{M} \right) \leq \frac{\mathbb{E}\left[ \sqrt{\langle V_t, \bold{1} \rangle} \right]}{\sqrt{\frac{\varepsilon}{M}}} \xrightarrow[t \to \infty]{} 0.
\end{equation}
\end{proof}
\noindent
We now extend the above result for integrable initial conditions.  
\begin{proposition}
\label{super_longtime_mass_of_duar}
    Let $\phi \in C_{\text{tem}}^+ \cap L^1(\mathbb{R})$ and let $(V_t)_{t \geq 0}$ be the unique $C_{\text{tem}}^+$-valued solution to $(\ref{eq: dual})$ with $V_0 = \phi$. Then, for all $f \in C_b^+$ we have $\langle V_t, f \rangle \xrightarrow[t \to \infty]{P} 0$. 
\end{proposition}
\begin{proof}
First, we fix some notation. Denote $M \equiv \lVert f \rVert_{\infty}$. For $g \in C_{\text{tem}}^+$, we denote by $(V_t^{g})_{t \geq 0}$ the unique $C_{\text{tem}}^+$-valued solution to $(\ref{eq: dual})$ with $V_0^g = g$. By linearity of $(\ref{eq: dual})$, we have that $V_t^{g+h} = V_t^g + V_t^h$ for all $g,h \in C_{\text{tem}}^+$ and $t \geq 0$. Let $\phi_n(x) \equiv \phi(x) e^{\frac{1}{x^2-n^2}}\mathds{1}_{(-n,n)}(x)$ and $\psi_n(x) \equiv \phi(x) - \phi_n(x)$. Fix arbitrary  $\varepsilon > 0$. Then, by the triangle inequality, we obtain the following for all $n \in \mathbb{N}$:

\begin{equation}
\label{prob_bound}
    \mathbb{P}\left(\langle V_t, f \rangle > \varepsilon \right) \leq \mathbb{P}\left(\langle V_t^{\phi_n}, \bold{1} \rangle > \frac{\varepsilon}{2M} \right) + \mathbb{P}\left(\langle V_t^{\psi_n}, \bold{1} \rangle > \frac{\varepsilon}{2M} \right).
\end{equation}
First, we bound the rightmost term in $(\ref{prob_bound})$.  
Since $\phi_n$ converges to $\phi$, and the convergence is monotone, we have by the Monotone Convergence Theorem that $\lim_{n \to \infty}\langle \phi_n, \bold{1} \rangle = \langle \phi, \bold{1} \rangle$, hence $\lim_{n \to \infty} \langle \psi_n, \bold{1} \rangle = 0$. Take $N$ such that $\langle \psi_N, \bold{1} \rangle < \frac{\varepsilon^2}{4M}$. Recall that by Lemma $\ref{lemma: dual_total_mass_mart}$, $(\langle V_t^{\psi_N}, \bold{1} \rangle)_{t \geq 0}$ is a martingale. Thus, for each $t \geq 0$, we have that:

\begin{equation}
\notag
    \mathbb{E}[\langle V_t^{\psi_N}, \bold{1} \rangle] = \mathbb{E}[\langle V_0^{\psi_N}, \bold{1} \rangle] =  \langle \psi_N, \bold{1} \rangle < \frac{\varepsilon^2}{4M}.
\end{equation}
It follows by the Markov inequality that:
\begin{equation}
\label{small_prob_1}
    \mathbb{P}\left(\langle V_t^{\psi_N}, \bold{1} \rangle > \frac{\varepsilon}{2M} \right) \leq \frac{\mathbb{E}\left[\langle V_t^{\psi_N}, \bold{1} \rangle \right]}{\frac{\varepsilon}{2M}} < \frac{\varepsilon}{2}.
\end{equation}
Since $\phi_N \in C_{\text{rap}}^+$, we can apply Proposition $\ref{longtime_mass_of_dual}$ to see that there exists $T \equiv T(\varepsilon)$, such that for all $t \geq T$:

\begin{equation}
\label{small_prob_2}
    \mathbb{P}\left(\langle V_t^{\phi_N}, \bold{1} \rangle > \frac{\varepsilon}{2M} \right) < \frac{\varepsilon}{2}.
\end{equation}
By $(\ref{small_prob_1})$, $(\ref{small_prob_2})$ and $(\ref{prob_bound})$ (with $n=N$), we conclude that for all $t \geq T$:

\begin{equation}
\notag
    \mathbb{P}\left(\langle V_t, f \rangle > \varepsilon \right) < \varepsilon,
\end{equation}
and since $\epsilon>0$ was arbitrary, we are done.
\end{proof}
\noindent
Next, we want to use the above result and the self-duality established in Corollary $\ref{cor: duality}$ to determine the longtime behavior of a solution to the parabolic Anderson model with bounded initial condition in $C_{\text{tem}}^+$, but not necessarily integrable. To that aim, we need to formulate a corresponding extended self-duality relation, which is stated in the following lemma.
\begin{lemma}
\label{super_duality_lemma}
Let $\dot{W}$ and $\dot{W_0}$ be two independent white noises on $\mathbb{R}_+ \times \mathbb{R}$. Let $\phi, \psi \in C_{\text{tem}}^+$. Define $(V_t)_{t \geq 0}$ and $(\Tilde{V}_t)_{t \geq 0}$ as the unique $C_{\text{tem}}^+$-valued solutions, of equation $(\ref{eq: dual})$, driven by $\dot{W}$ and $\dot{W_0}$, respectively, with initial conditions $V_0 = \phi$ and $\Tilde{V}_0 = \psi$. Then, for all $t \geq 0$:
\begin{equation}
\notag
    \mathbb{E}\left[e^{-\langle V_t, \psi \rangle}\right] = \mathbb{E}\left[ e^{-\langle \phi, \Tilde{V}_t \rangle} \right].
\end{equation}
\end{lemma}
\begin{proof}
    Let $\phi_n(x) \equiv \phi(x) e^{\frac{1}{x^2-n^2}}\mathds{1}_{(-n,n)}(x)$ and let $(V_t^n)_{t \geq 0}$ be the corresponding unique $C_{\text{tem}}^+$-valued solution to $(\ref{eq: dual})$ with $V_0^n = \phi_n$.  Since $\phi_n \in C_{\text{rap}}^+$, we can use Corollary $\ref{cor: duality}$ to get:
    \begin{equation}
    \label{nth duality}
    \notag
        \mathbb{E}\left[e^{-\langle V_t^n, \psi \rangle}\right] = \mathbb{E}\left[ e^{-\langle  \phi_n, \Tilde{V}_t \rangle} \right].
    \end{equation}
    Now be letting $n \to \infty$, the result follows by monotone convergence of $V^n_t$ to $V_t$ and $\phi_n$ to $\phi$. We leave the details to the reader to complete.

\end{proof}
\begin{remark}
Note that $\langle V_t, \psi \rangle$ and $\langle \phi, \Tilde{V}_t \rangle$ may take infinite values. 
\end{remark}
\noindent
Proposition $\ref{super_longtime_mass_of_duar}$ and Lemma $\ref{super_duality_lemma}$ give us all we need to prove the following result regarding the longtime behavior of the parabolic Anderson model with bounded initial conditions. 

\begin{corollary}
\label{cor: ref I didn't find}
    Let $\dot{W}$ be a white noise on $\mathbb{R}_+ \times \mathbb{R}$.
    Let $f \in C_b^+$ and let $(V_t)_{t \geq 0}$ be the unique $C_{\text{tem}}^+$-valued solution to $(\ref{eq: dual})$, driven by $\dot{W}$, with initial condition $V_0 = f$. Let $g \in C_{\text{tem}}^+ \cap L^1(\mathbb{R})$. Then, $\langle V_t, g \rangle \xrightarrow[t \to \infty]{P} 0$.
\end{corollary}
\begin{proof}
    Let $\dot{W_0}$ be a white noise on $\mathbb{R}_+ \times \mathbb{R}$, independent of $\dot{W}$. Let $(\Tilde{V}_t)_{t \geq 0}$ be the unique $C_{\text{tem}}^+$-valued solution to $(\ref{eq: dual})$, driven by $\dot{W_0}$, with $\Tilde{V}_0 = g$, independent of $(V_t)_{t \geq 0}$. By Lemma $\ref{super_duality_lemma}$, we get that:

    \begin{equation}
    \label{tadam: Laplace_is_the_same}
        \mathbb{E}\left[e^{-\langle V_t, g \rangle} \right] = \mathbb{E}\left[e^{-\langle f, \Tilde{V}_t \rangle} \right].
    \end{equation}
    By Proposition 3.11, $\langle \Tilde{V}_t, f \rangle \xrightarrow[t \to \infty]{P} 0$, and the result follows from $(\ref{tadam: Laplace_is_the_same})$.
    
\end{proof}
\noindent
We have gathered all that we needed to prove Theorem $\ref{main_res: global_longtime}$ and Theorem $\ref{main_res: local_longtime}$. Let us start with the proof of Theorem $\ref{main_res: global_longtime}$.
\begin{proof}[Proof of Theorem 2.9]
     Assume without loss of generality that $\langle Y_0, \bold{1} \rangle = \langle \phi - \psi, \bold{1} \rangle \geq 0$. Let $\varepsilon > 0$. Since $Y_t =S_t(Y_0)$, by Corollary $\ref{cor: negative part vanishes}$, there exists $\hat{T}$ such that $\langle Y_{t+\hat{T}}^-, 1 \rangle < \varepsilon$ for all $t \geq 0$. Let $a > 0$ and let $(V_t)_{t \geq 0}$ be the unique $C_{\text{tem}}^+$-valued solution of $(\ref{eq: dual})$ with $V_0 = a\bold{1}$. By Lemma $\ref{lemma: dual_total_mass_mart}$, we get that $\left(-\langle S_{T-t}V_t, \lvert Y_{\hat{T}} \rvert \rangle\right)_{t \in [0,T]}$ is a continuous square integrable martingale with quadratic variation $\int_0^t \langle V_s^2, (S_{T-s}\lvert Y_{\hat{T}} \rvert)^2 \rangle \mathrm{d}s$. Therefore, $\left(e^{-\langle S_{T-t}V_t, \lvert Y_{\hat{T}} \rvert \rangle - \frac{1}{2}\int_0^t \langle V_s^2, (S_{T-s}\lvert Y_{\hat{T}} \rvert)^2 \rangle \mathrm{d}s}\right)_{t \in [0,T]}$ is a local martingale. Note that it is bounded almost surely, hence it is a martingale. It follows that for all $T > 0$ and $t \in [0,T]$:
\begin{equation}
\notag
    \mathbb{E}\left[e^{-\langle S_{T-t}V_t, \lvert Y_{\hat{T}} \rvert \rangle - \frac{1}{2}\int_0^t \langle V_s^2, (S_{T-s}h)^2 \rangle \mathrm{d}s} \right] = \mathbb{E}\left[e^{-\langle S_{T}V_0, \lvert Y_{\hat{T}} \rvert \rangle} \right] =  e^{-a\langle \bold{1}, \lvert Y_{\hat{T}} \rvert \rangle}.
\end{equation}
In particular, for $t = T$ we get:

\begin{equation}
\label{equality that helps}
    \mathbb{E}\left[e^{-\langle V_T, \lvert Y_{\hat{T}} \rvert \rangle - \frac{1}{2}\int_0^T \langle V_s^2, (S_{T-s}\lvert Y_{\hat{T}} \rvert)^2 \rangle \mathrm{d}s} \right] = e^{-a\langle \bold{1}, \lvert Y_{\hat{T}} \rvert \rangle}.
\end{equation}
\noindent
By Lemma $\ref{sol_is_mart}$, $U_{\hat{T}}^i \in L^1(\mathbb{R})$, for $i=1,2$. Therefore, since $X = U_1 + U_2$ and $Y = U_1 - U_2$, we have that $X_{\hat{T}}, \lvert Y_{\hat{T}} \rvert \in L^1(\mathbb{R}).$ Thus, we can apply Corollary $\ref{cor: ref I didn't find}$ to get:

\begin{equation}
\label{limits of ref I didn't find}
    \lim_{T \to \infty}\langle V_T, X_{\hat{T}} \rangle = \lim_{T \to \infty}\langle V_T, \lvert Y_{\hat{T}} \rvert \rangle = 0 \quad \text{ in probability.}
\end{equation}
\noindent
Combine $(\ref{equality that helps})$ and $(\ref{limits of ref I didn't find})$ to see that:

\begin{equation}
\label{importnat_convergence}
    \begin{aligned}
        \lim_{T \to \infty} \mathbb{E}\left[e^{-\langle V_T, X_{\hat{T}} \rangle- \frac{1}{2}\int_0^T \langle V_s^2, (S_{T-s}|Y_{\hat{T}}|)^2 \rangle \mathrm{d}s} \right]  = \mathbb{E}\left[e^{-a\langle \bold{1}, |Y_{\hat{T}}| \rangle} \right].
    \end{aligned}
\end{equation}
\noindent
Note that $(S_{T-s}Y_{\hat{T}})^2 \leq (S_{T-s}|Y_{\hat{T}}|)^2 $ and thus:
\begin{equation}
\label{important_ineqality}
    \mathbb{E}\left[e^{-\langle V_T, X_{\hat{T}} \rangle- \frac{1}{2}\int_0^T \langle V_s^2, (S_{T-s}Y_{\hat{T}})^2 \rangle \mathrm{d}s} \right] \geq \mathbb{E}\left[e^{-\langle V_T, X_{\hat{T}} \rangle- \frac{1}{2}\int_0^T \langle V_s^2, (S_{T-s}|Y_{\hat{T}}|)^2 \rangle \mathrm{d}s} \right].
\end{equation}

\noindent
It is easy to see that since $(X_t)_{t \geq 0}$ satisfies the martingale problem $(MP)_X$, the process $(X_{t + \hat{T}})_{T \geq 0}$ also satisfies the same martingale problem, with initial condition $X_{\hat{T}}$. Hence we can use the previously established duality formula $(\ref{duality_formula})$ to get:

\begin{equation}
\label{application_of_duality}
    \mathbb{E}\left[e^{-\langle V_T, X_{\hat{T}} \rangle- \frac{1}{2}\int_0^T \langle V_s^2, (S_{T-s}Y_{\hat{T}})^2 \rangle \mathrm{d}s} \right] = \mathbb{E}\left[e^{-a\langle \bold{1}, X_{\hat{T}+T} \rangle} \right].
\end{equation}

\noindent
Therefore, using $(\ref{importnat_convergence})$, $(\ref{important_ineqality})$ and $(\ref{application_of_duality})$, we get the lower bound:
\begin{equation}
\label{lower bound}
    \liminf_{T \to \infty}\mathbb{E}\left[e^{-a\langle  \bold{1} , X_{\hat{T}+T}\rangle} \right] \geq \mathbb{E}\left[e^{-a\langle \bold{1}, |Y_{\hat{T}}| \rangle} \right] = \mathbb{E}\left[e^{-a\langle \bold{1}, Y_{\hat{T}} \rangle - 2a \langle \bold{1}, Y_{\hat{T}}^- \rangle} \right] \geq e^{-2a\varepsilon} e^{-a \langle \bold{1}, Y_0  \rangle},
\end{equation}
where the last inequality follows from the choice of $\hat{T}$ and part $(a)$ of Proposition $\ref{heat_longtime_1}$. \\ \\
For an upper bound, by definition of $X$ and $Y$, we have that $X_T \geq Y_T$ for all $T \geq 0$. So for all $T \geq 0$:
\begin{equation}
\label{upper bound}
    \mathbb{E}\left[e^{-a\langle \bold{1}, X_{T} \rangle} \right] \leq \mathbb{E}\left[ e^{-a \langle \bold{1}, Y_T  \rangle} \right] =  e^{-a \langle \bold{1}, Y_0  \rangle},
\end{equation}
where we used again Proposition $\ref{heat_longtime_1}$.
\noindent
Combine the bounds in $(\ref{upper bound})$ and $(\ref{lower bound})$ to see that:

\begin{equation}
\label{double_bound}
    \begin{aligned}
        e^{-a \langle Y_0,  \bold{1}  \rangle} &\geq \limsup_{T \to \infty} \mathbb{E}\left[e^{-a\langle X_{T}, \bold{1} \rangle} \right] \\ &\geq \liminf_{T \to \infty} \mathbb{E}\left[e^{-a\langle X_{T}, \bold{1} \rangle} \right] \\ &= \liminf_{T \to \infty}\mathbb{E}\left[e^{-a\langle X_{\hat{T}+T}, \bold{1} \rangle} \right] \\ &\geq e^{-2a\varepsilon} e^{-a \langle \bold{1}, Y_0  \rangle}.
    \end{aligned}
\end{equation}
\noindent
Taking $\varepsilon \to 0$ in $(\ref{double_bound})$ we have that:

\begin{equation}
    \label{X converges!}
    \lim_{T \to \infty}\mathbb{E}\left[e^{-a\langle X_{T}, \bold{1} \rangle} \right] = e^{-a \langle  Y_0,  \bold{1}  \rangle}.
\end{equation}
\noindent
Since $a > 0$ was arbitrary, $(\ref{X converges!})$ states the convergence of the one-sided Laplace transform of $\langle X_T , \bold{1} \rangle$ to that of $\langle Y_0 , \bold{1} \rangle$. It follows that $\langle X_T, \bold{1} \rangle \xrightarrow[T \to \infty]{\mathcal{D}}  \langle Y_0,  \bold{1} \rangle$, and since the limit is deterministic, the convergence holds in probability as well. By Lemma $\ref{sol_is_mart}$, and since $X = U_1 + U_2$, we get that $\left(\langle X_T, \bold{1} \rangle\right)_{T \geq 0}$ is a non-negative square integrable martingale. Therefore, by the martingale convergence theorem, it converges almost surely. By uniqueness of the limit, we get:

\begin{equation}
\notag
    \lim_{T \to \infty} \langle X_T , \bold{1} \rangle = \langle Y_0 , \bold{1} \rangle \quad \text{a.s.}
\end{equation}
It follows that:
\begin{equation}
\notag
\begin{aligned}
\langle U^2_T, \bold{1} \rangle &= \frac{1}{2}\left(\langle X_T, \bold{1} \rangle - \langle Y_T, \bold{1} \rangle \right) =  \frac{1}{2}\left(\langle X_T, \bold{1} \rangle - \langle Y_0, \bold{1} \rangle \right)  \xrightarrow[T \to \infty]{\text{a.s.}} 0, \\
\langle U^1_T, \bold{1} \rangle &=  \frac{1}{2}\left(\langle X_T, \bold{1} \rangle + \langle Y_T, \bold{1} \rangle \right) =  \frac{1}{2}\left(\langle X_T, \bold{1} \rangle + \langle Y_0, \bold{1} \rangle \right)  \xrightarrow[T \to \infty]{\text{a.s.}}  \langle Y_0, \bold{1} \rangle  = \langle \phi - \psi, \bold{1} \rangle.
\end{aligned}
\end{equation}
Therefore

\begin{equation}
\notag
    \mathbb{P}\left(\lim_{T \to \infty} \langle U^1_T, \bold{1} \rangle \langle U^2_T, \bold{1} \rangle = 0 \right) = 1,
\end{equation}
and we are done.
\end{proof}
\noindent
We proceed to the proof of Lemma $\ref{local_coexistence_is_ok}$, which shows that the definition of local coexistence makes proper sense. We then continue to prove Theorem $\ref{main_res: local_longtime}$, which states local coexistence is not possible when the initial conditions are continuous, bounded and have a well-defined mean.
\begin{proof}[Proof of Lemma 2.7]
 Let $B \in \mathscr{B}_b$, and let $\varphi \in C_{\text{rap}}^+$ be such that $\varphi \geq 1$ on $B$. Thus,
 \begin{equation}
 \notag
 \begin{aligned}
\mathbb{E}\left[\int_{B} U^i_t(x) \mathrm{d}x \right] &\leq \mathbb{E}\left[ \langle X_t, \varphi \rangle \right] = \langle S_T X_0, \varphi \rangle \leq \lVert X_0 \rVert_{\infty} \langle 1, \varphi \rangle < \infty,
\end{aligned}
 \end{equation}
 where we used the martingale problem $(MP)_X$ and Proposition $\ref{uniform_bound_heat}$. By the Markov inequality and Theorem 4.10 in \cite{MR3642325}, we are done.
\end{proof}
\begin{proof}[Proof of Theorem 2.10]
Recall that $Y_0 \in \mathcal{M}$ so that $$\overline{Y}_0 = \lim_{L \to \infty} \frac{1}{2L}\int_{-L}^L Y_0(x) \mathrm{d}x$$ is well-defined. Assume without loss of generality that $\overline{Y}_0 \geq 0$. 
Let $f \in C_{\text{rap}}^{+}$, and let $(V_t)_{t \geq 0}$ be the unique $C_{\text{rap}}^+$-valued solution of $(\ref{eq: dual})$ with initial condition $V_0 = f$. Note that by Lemma $\ref{lemma: dual_total_mass_mart}$, we get that $-\langle V_t, \overline{Y}_0\bold{1} \rangle$ is a continuous square integrable martingale with quadratic variation $\int_0^t \left \langle V_s^2, \left( \overline{Y}_0 \right)^2 \bold{1} \right \rangle \mathrm{d}s$. Therefore, $e^{-\langle V_t, \overline{Y}_0\bold{1} \rangle - \frac{1}{2}\int_0^t \langle V_s^2, \left( \overline{Y}_0 \right)^2 \bold{1} \rangle \mathrm{d}s}$ is a local martingale. Note it is bounded almost surely, hence it is a martingale. It follows that for all $t \geq 0$:

\begin{equation}
\label{eq: exp_mart_property}
    \mathbb{E}\left[e^{-\langle V_t, \overline{Y}_0\bold{1} \rangle - \frac{1}{2}\int_0^t \langle V_s^2, \left( \overline{Y}_0 \right)^2 \bold{1} \rangle \mathrm{d}s}\right] = e^{-\overline{Y}_0\langle f, \bold{1} \rangle}.
\end{equation}
\noindent
By Proposition $\ref{longtime_mass_of_dual}$, we know that $\lim_{t \to \infty}\langle V_t, \overline{Y}_0 \bold{1} \rangle = 0$ in probability, so taking the limit as $t \to \infty$ in Equation $(\ref{eq: exp_mart_property})$ yields:

\begin{equation}
\label{eq: finite_limit_for_dct}
    \mathbb{E}\left[e^{-\frac{1}{2}\int_0^{\infty}\langle V_s^2 , \left( \overline{Y}_0 \right)^2 \bold{1} \rangle \mathrm{d}s} \right] = e^{-\overline{Y}_0\langle f, \bold{1} \rangle}.
\end{equation}
\noindent
Let $(T_n)_{n \geq 1}$ and $(\hat{T}_n)_{n \geq 1}$ be non-negative sequences such that $\lim_{n \to \infty}T_n = \lim_{n \to \infty}\hat{T}_n = \infty.$ Note that for each $n \geq 1$, the process $(X_{\hat{T}_n + t})_{t \geq 0}$ solves the martingale problem $(MP)_X$. Hence, we can use our duality formula from Proposition $\ref{prop: duality}$ to get for all $n \geq 1$:

\begin{equation}
\label{advanced_duality_formula}
    \mathbb{E}\left[e^{-\langle X_{\hat{T}_n+T_n}, f \rangle}\right] = \mathbb{E}\left[e^{-\langle X_{\hat{T}_n}, V_{T_n} \rangle -\frac{1}{2}\int_0^{T_n} \langle V_s^2, Y_{\hat{T}_n+T_n-s}^2 \rangle \mathrm{d}s}\right].
\end{equation}
We wish to take the limit as $n \to \infty$ in $(\ref{advanced_duality_formula})$. First, we treat the convergence of the term $\int_0^{T_n} \langle V_s^2, Y_{\hat{T}_n+T_n-s}^2 \rangle \mathrm{d}s$. By Proposition $\ref{convergence_of_bounded_heat}$, $\lim_{n \to \infty}Y_{\hat{T}_n+T_n-s}^2(x) = \left( \overline{Y}_0 \right)^2$ for all $x \in \mathbb{R}$. Recall that $V_s\in C_{\text{rap}}$, and moreover, by Proposition $\ref{uniform_bound_heat}$,  $V_s^2(\cdot) Y_{\hat{T}_n+T_n-s}^2(\cdot) \leq \lVert Y_0 \rVert_{\infty}^2 V_s^2(\cdot) \in L^1(\mathbb{R})$, almost surely. Thus by dominated convergence we obtain:

\begin{equation}
\label{conv. }
    \lim_{n \to \infty}\langle V_s^2, Y_{\hat{T}_n+T_n-s}^2 \rangle = \langle V_s^2, \left( \overline{Y}_0 \right)^2\bold{1} \rangle, \quad \text{a.s. for each $s \geq 0$}.
\end{equation}
By taking $\phi \equiv \bold{1}$ in the martingale problem $(MP)_V$, we get that $\langle V_t, \bold{1} \rangle$ is a square integrable martingale with quadratic variation $\int_0^t \langle V_s^2, \bold{1} \rangle \mathrm{d}s.$ 
By Proposition $\ref{longtime_mass_of_dual}$, $\langle V_t, \bold{1} \rangle$ converges to $0$ in probability as $t \to \infty$. By the martingale convergence theorem, the convergence holds almost surely as well. By Proposition 1.8 in \cite{MR1725357}, we get that the quadratic variation of $\langle V_t, \bold{1} \rangle$ has to be almost surely finite, that is

\begin{equation}
\label{finite_quad_var}
    \int_0^\infty  \langle V_s^2, \bold{1} \rangle \mathrm{d}s < \infty \quad \text{ almost surely.}
\end{equation}
\noindent
Use again $\lVert Y_s(\cdot) \rVert_{\infty} \leq \lVert Y_0\rVert_{\infty}$, for all $s \geq 0$, to get:
\begin{equation}
\label{bound_L^1}
    \langle V_s^2,  Y_{\hat{T}_n+T_n-s}^2 \rangle \mathds{1}_{s \leq T_n} \leq \lVert Y_0 \rVert_{\infty}^2 \langle V_s^2, \boldsymbol{1} \rangle.
\end{equation}
Thus, $(\ref{finite_quad_var}), (\ref{bound_L^1}), (\ref{conv. })$ and the Dominated Convergence Theorem imply that:
\begin{equation}
\begin{aligned}
\label{convergence_of_iint}
\lim_{n \to \infty} \int_0^{T_n}\langle V_s^2,  Y_{\hat{T}_n+T_n-s}^2 \rangle \mathrm{d}s = \int_0^{\infty} \langle V_s^2, \left( \overline{Y}_0 \right)^2\bold{1} \rangle \mathrm{d}s \quad \text{ almost surely.}
\end{aligned}
\end{equation}
Now we will establish the limit of $\langle X_{\hat{T}_n}, V_{T_n} \rangle$. Note that by conditional Jensen's inequality we get:
    \begin{equation}
    \notag
        \begin{aligned}
            \mathbb{E}\left[\sqrt{\langle X_{\hat{T}_n}, V_{T_n}} \rangle\right] &= \mathbb{E}\left[ \mathbb{E}\left[ \sqrt{\langle X_{\hat{T}_n}, V_{T_n}} \rangle | V_{T_n} \right] \right] \leq \mathbb{E}\left[ \sqrt{\mathbb{E}\left[ \langle X_{\hat{T}_n}, V_{T_n} \rangle | V_{T_n} \right]} \right] \\ &= \mathbb{E}\left[ \sqrt{ \langle S_{\hat{T}_n}X_{0}, V_{T_n} \rangle} \right] \leq \sqrt{\lVert X_0 \rVert_{\infty}} \mathbb{E}\left[ \sqrt{\langle V_{T_n}, \bold{1} \rangle} \right] \xrightarrow[n \to \infty]{} 0,
        \end{aligned}
    \end{equation}
    where the convergence follows from the last display in the proof of Proposition 4.3 in \cite{MR4700256}. Thus, an application of the Markov inequality implies that 
    
    \begin{equation}
    \label{convergence_of_first_term}
        \langle X_{\hat{T}_n}, V_{T_n} \rangle \xrightarrow[n \to \infty]{P} 0.
    \end{equation}
\noindent
$(\ref{convergence_of_iint})$ and $(\ref{convergence_of_first_term})$ imply that:
\begin{equation}
\label{limit_of_RHS_of_duality}
\lim_{n \to \infty} \mathbb{E}\left[e^{-\langle X_{\hat{T}_n}, V_{T_n} \rangle -\frac{1}{2}\int_0^{T_n} \langle V_s^2, Y_{\hat{T}_n+T_n-s}^2 \rangle \mathrm{d}s}\right] = \mathbb{E}\left[e^{-\frac{1}{2}\int_0^{\infty}\langle V_s^2 ,\left( \overline{Y}_0 \right)^2 \bold{1} \rangle \mathrm{d}s} \right].
\end{equation}
Combine $(\ref{eq: finite_limit_for_dct}),(\ref{advanced_duality_formula}) $, and $(\ref{limit_of_RHS_of_duality})$ to get:

\begin{equation}
\notag
\label{MGF_convergence}
    \lim_{n \to \infty} \mathbb{E}\left[e^{-\langle X_{\hat{T}_n+T_n}, f \rangle}\right] = e^{-\overline{Y}_0\langle f, \bold{1} \rangle}.
\end{equation}
Thus, we can easily conclude that $\langle X_t, f \rangle \xrightarrow[t \to \infty]{\mathcal{D}}  \langle \overline{Y}_0 \bold{1}, f \rangle$ for all $f \in C_{\text{rap}}^{+}$. By Proposition $\ref{convergence_of_bounded_heat}$, it follows that $\langle Y_t, f \rangle \xrightarrow[t \to \infty]{\mathcal{D}}  \langle  \overline{Y}_0\bold{1}, f \rangle$ for all $f \in C_{\text{rap}}^{+}$. The limits are deterministic, so the convergences hold also in probability. Therefore:

\begin{equation}
\label{C_rap_convergence}
\begin{aligned}
    \langle U^2_t, f \rangle &= \left \langle \frac{X_t - Y_t}{2}, f \right \rangle \xrightarrow[t \to \infty]{P} 0, \quad \forall f \in C_{\text{rap}}^+. 
\end{aligned}
\end{equation}
\noindent
We want to extend the convergence in $(\ref{C_rap_convergence})$ for functions $g \in C_{\text{exp}}.$ By linearity, it is enough to consider functions in $C_{\text{exp}}^+$. Let $g \in C_{\text{exp}}^+$, and let $f_n(x) \equiv g(x) e^{\frac{1}{x^2-n^2}}\mathds{1}_{(-n,n)}(x)$. Note that $f_n \in C_{\text{rap}}^+$ and that $f_n \uparrow g$. By the martingale problem $(MP)_X$, we get that for each $n$, $\left(\langle S_{T-t}X_t, f_n \rangle \right)_{t \in [0,T]}$ is a martingale, and in particular $\mathbb{E}\left[\langle X_t, f_n \rangle \right] = \langle S_tX_0, f_n \rangle$. Taking the limit as $n \to \infty$ we get that $\mathbb{E}\left[\langle X_t, g \rangle \right] = \langle S_tX_0, g \rangle$ by the Monotone Convergence Theorem. It follows that for all $\varepsilon > 0$:
\begin{equation}
\notag
    \begin{aligned}
        \mathbb{P}\left( \langle U^2_t, g-f_n \rangle > \varepsilon \right) &\leq \frac{\mathbb{E}[\langle U^2_t, g-f_n \rangle ]}{\varepsilon} \\ &\leq \frac{\mathbb{E}[\langle X_t, g-f_n \rangle ]}{\varepsilon} \\ &= \frac{\langle S_t X_0, g-f_n \rangle}{\varepsilon}  \\ & \leq \frac{\lVert X_0 \rVert_{\infty} \langle \bold{1}, g-f_n \rangle}{\varepsilon} \xrightarrow[]{n \to \infty} 0,
    \end{aligned}
\end{equation}
where the last convergence follows by montone convergence. Fix arbitrary $\varepsilon > 0$. Then, there exists $N$ such that $\mathbb{P}\left( \langle U^2_t, g-f_n \rangle > \varepsilon \right) < \varepsilon$ for all $n \geq N$ and $t \geq 0$. By $(\ref{C_rap_convergence})$, there exists $T = T(\varepsilon)$ such that for all $t \geq T$ we have that $\mathbb{P}\left( \langle U^2_t, f_N \rangle > \varepsilon \right) \leq \varepsilon$. We get that for all $t \geq T$:
\begin{equation}
\notag
    \begin{aligned}
        \mathbb{P}\left( \langle U^2_t, g \rangle > 2\varepsilon \right) &\leq
        \mathbb{P}\left( \langle U^2_t, g-f_N \rangle > \varepsilon \right) + \mathbb{P}\left( \langle U^2_t, f_N \rangle > \varepsilon \right) \leq 2\varepsilon
    \end{aligned}
\end{equation}
Since $\varepsilon > 0$ was arbitrary, it follows that $\langle U^2_t, g \rangle \xrightarrow[t \to \infty]{P} 0$. By Proposition $\ref{convergence_of_bounded_heat}$, it follows that $\langle Y_t, g \rangle \xrightarrow[t \to \infty]{P}  \overline{Y}_0\langle  \bold{1}, g \rangle$ and since $U^1_t = Y_t + U^2_t$, we get that:

\begin{equation}
\notag
\label{C_exp_convergence}
\begin{aligned}
    \langle U^1_t, g \rangle \xrightarrow[t \to \infty]{P}  \overline{Y}_0\langle \bold{1}, g \rangle, \quad \forall g \in C_{\text{exp}}, \\
    \langle U^2_t, g \rangle \xrightarrow[t \to \infty]{P} 0, \quad \forall g \in C_{\text{exp}}. 
\end{aligned}
\end{equation}
This completes the proof.
\end{proof}
\section{Acknowledgments}
We thank Zhenyao Sun, Johanna Weinberger, and Yam Bernet for very helpful conversations. The work of the authors was supported in part by ISF Grant
1985/22. 
\bibliographystyle{plain}
\bibliography{sample}

\end{document}